\documentclass[journal]{IEEEtran}
\usepackage{amssymb}
\usepackage{graphicx}
\usepackage{bbm}
\usepackage{epsfig}
\usepackage{graphics,graphicx,amssymb,amsmath,verbatim}
\usepackage{mathrsfs}
\usepackage{amsfonts}
\usepackage{color}

\newtheorem{theorem}{Theorem}

\newtheorem{corollary}{Corollary}

\newtheorem{problem}{Problem}

\newtheorem{remark}{Remark}

\newenvironment{proof}[1][Proof]{\noindent\textbf{#1.} }{\ \rule{0.5em}{0.5em}}
\ifCLASSINFOpdf
\else
\fi
\input{tcilatex}
\begin{document}

\title{The Noise Covariances of Linear Gaussian Systems with Unknown Inputs Are Not Uniquely Identifiable Using Autocovariance Least-squares}
\author{He~Kong,~Salah~Sukkarieh,~Travis J. Arnold,~Tianshi Chen,
and~Wei Xing Zheng \thanks{\textbf{This paper is accepted and going to be officially published in \emph{Systems \& Control Letters}.} Correspond author: He Kong.}\thanks{%
He Kong is with the Department of Mechanical and Energy Engineering, South University of Science and Technology, Shenzhen, 518055, China; email:
kongh@sustech.edu.cn. Salah Sukkariehb is with the Australian Centre for Field Robotics, University of Sydney, NSW 2006, Australia; email: salah.sukkarieh@sydney.edu.au. Travis J. Arnold is an independent researcher; email: travis.arnold17@gmail.com. Tianshi Chen is with the School of Data Science and Shenzhen Research Institute of Big Data, The Chinese University of Hong Kong, Shenzhen, 518172, China; email: tschen@cuhk.edu.cn. Wei Xing Zheng is with the School of Computer, Data and Mathematical Sciences, Western Sydney University, Sydney, NSW 2751, Australia; email: w.zheng@westernsydney.edu.au.}}
\maketitle

\begin{abstract}
Existing works in optimal filtering for linear Gaussian
systems with arbitrary
unknown inputs assume
perfect knowledge of the noise covariances
in the filter design. This is
impractical and raises the question of whether and under what conditions
one can identify the
noise covariances of linear Gaussian systems with
arbitrary unknown inputs. This paper considers the above identifiability
question using the correlation-based autocovariance least-squares (ALS)
approach. In particular, for the ALS framework, we prove that (i) the process noise covariance $Q$ and the measurement noise covariance $R$
cannot be uniquely jointly identified; (ii) neither $Q$ nor $R$ is uniquely identifiable, when the other is known. This not
only helps us to
have a better understanding of the applicability of
existing filtering
frameworks under unknown inputs (since almost all of them
require perfect
knowledge of the noise covariances) but also calls for
further
investigation of alternative and more viable noise covariance
methods under
unknown inputs. Especially, it remains to be explored whether the noise covariances are uniquely identifiable using other correlation-based methods. We are also interested to use regularization for
noise covariance estimation under unknown inputs, and investigate the relevant property guarantees for the covariance estimates. The above topics are the main subject of our current and future work.
\end{abstract}



\markboth{}{Shell \MakeLowercase{\textit{et al.}}: A Divide and Conquer Approach to Cooperative Distributed Model
Predictive Control}




\IEEEpeerreviewmaketitle

\section{Introduction}

The last few decades have witnessed much progress in the development of
parameter identification techniques and their applications in practice (see,
e.g., \cite{Mahata2004}-\cite{Mu2014}). Popular state estimation methods
include the well-known Kalman filter (KF) and adaptive KF \cite%
{Gustafsson2000}, moving horizon estimation \cite{Rawlings2001}-\cite%
{Alessandri2020}, etc. Despite existing methods' versatility, their
performance might be compromised under unmodeled dynamics whose models or
statistical properties are hard to obtain. Typical scenarios include system
faults, abrupt/jumping noises, arbitrary vehicle tires/ground interactions,
and systems with network-induced effects or attacks (see the in-depth
discussions in \cite{Johansen2014}-\cite{Yang2019} and the references
therein). Hence, estimation under arbitrary unknown inputs (whose models or
statistical properties are not assumed to be available), also called unknown
input decoupled estimation, has received much attention in the past. A
seminal work on unknown input decoupled estimation is due to Hautus \cite%
{Hautus1983} where it has been shown that the strong detectability\footnote{%
The strong$^{\ast }$ detectability concept was also introduced in \cite%
{Hautus1983}. The two criteria, as discussed in \cite{Hautus1983}, are
equivalent for discrete-time systems, but differ for continuous systems.
Here we focus on the filtering case of discrete-time systems.} criterion is
necessary and sufficient for the existence of a stable observer for
estimating the state/unknown input.

Works on the filtering case, e.g., \cite{Darouach2003}-\cite{Su2015}, have
similar rank matching and system being minimum phase requirements as in \cite%
{Hautus1983}. Extensions to cases with rank-deficient shaping matrices have
been discussed in \cite{Hsieh2009}-\cite{Frazzoli2016}. It has also been
shown in the above works that for unbiased and minimum variance estimation
of the state/unknown input, the initial guess of the state must be unbiased.
Recently, connections between the above-mentioned results and KF of systems
within which the unknown input is taken to be a white noise of unbounded
variance, have been established in \cite{Bitmead2019}. There are also works
dedicated to alleviating the strong detectability conditions and the
unbiased initialization requirement, and the incorporation of norm
constraints (see \cite{Kong2019Auto}-\cite{Kong2021Auto} and the references
therein).

However, most existing filtering works mentioned above assume that the
process and measurement noise covariances (denoted as $Q$ and $R$,
respectively) are perfectly known for the optimal filter design. This raises
the question of whether and under what conditions one can identify $Q/R$
from real data. We believe that addressing the identifiability issue of
noise covariances under arbitrary unknown inputs is important because in
practice the noise covariances are not known \textit{a priori} and have to
be identified from real closed-loop data where there might be unknown system
uncertainties such as faults, etc. Another relevant application is tracking
of targets whose motions might be subject to abrupt disturbances (in the
form of unknown inputs), as considered in our recent work \cite{Kong2021ICRA}%
.

To our best knowledge, \cite{Yu2016}-\cite{Pan2013} are the only existing
works on identification of stochastic systems under unknown inputs. However,
in \cite{Yu2016}-\cite{Pan2013}, the unknown inputs are assumed to be a
wide-sense stationary process with rational power spectral density or
deterministic but unknown signals. We do not make such assumptions here.
Also, we aim to investigate the identifiability of the original noise
covariances for linear Gaussian systems with unknown inputs. This is in
contrast to the work in \cite{Yu2016} where the measurement noise covariance
of the considered system is assumed to be known, and the input
autocorrelations are identified from the output data and then used for input
realization and filter design. Our work is also different from subspace
identification where the stochastic parameters of the system are estimated
and used to calculate the optimal estimator gain \cite{Gevers2006}.

Note that noise covariance estimation is a special parameter identification
question that is of lasting interest for the control community, and the
literature is fairly mature. Existing noise covariance estimation methods
can be classified as Bayesian, maximum likelihood, covariance matching,
correlation-based techniques, etc., (see \cite{Garatti2014}-\cite%
{Kerrigan2017} and the references therein). Especially, the autocovariance
least-squares (ALS) framework in \cite{Rawlings2006}-\cite{Arnold2020} is a
popular correlation-based method that has gained much attention in the
recent literature \cite{Ge2017}-\cite{Deng2020}. The main concept of ALS is
to design a stable filter without having to know the true noise covariances.
Since the filter is suboptimal, the innovations will be correlated, based on
which the noise covariance estimation question can be transformed to a
standard least-squares optimization problem.

Still, most noise covariance estimation methods mentioned-above have not
considered the case with unknown inputs. This observation motivates us to
study the identifiability of $Q/R$ for systems under unknown inputs.
Especially, we discuss the correlation-based ALS framework and show that (i)
to apply the ALS framework for the problem at hand, one has to apply a
linear transformation to the innovation so that it is decoupled from the
unknown inputs (see discussions in Section III. A); (ii) the ALS problem for
jointly estimating $Q$ and $R$ does not have a unique solution (see in
Theorem \ref{ALS_solution}); (iii) the ALS problem for estimating $Q$ or $R$%
, when the other is known, does not have a unique solution (see Corollary %
\ref{Q_unknown_ALS_extension}).

The above findings reveal that the noise covariances are in general not
uniquely identifiable using the ALS approach. This not only helps us to
better understand the applicability of existing filtering frameworks under
unknown inputs (since almost all of them require perfect knowledge of the
noise covariances) but also calls for further investigation of alternative
and more viable noise covariance methods under unknown inputs. Especially,
it remains to be explored whether the noise covariances are uniquely
identifiable using other correlation-based methods (see \cite{DunikSurvey}
and the references therein). We are also interested to use regularization for
noise covariance estimation under unknown inputs, and investigate the relevant property guarantees for the covariance estimates (see \cite{Chen2013}
and the references therein). The above topics are the main subject of our
current and future work.

The remainder of the paper is organized as follows. In Section II, we recall
preliminaries on estimation of systems with unknown inputs. Section III
contains our major results. Section IV verifies the theoretical results with
numerical examples. Section V concludes the paper. \textbf{Notation:} $A^{%
\mathrm{T}}$ denotes the transpose of matrix $A$. $\mathbf{R}^{n}$ stands
for the $n$-dimensional Euclidean space. $I_{n}$ stands for identity
matrices of $n$ dimension. $\mathbb{C}$ and $\left\vert z\right\vert $
denote the field of complex numbers and the absolute value of a given
complex number $z$, respectively. $[a_{1},\cdots ,a_{n}]$ denotes $[a_{1}^{%
\mathrm{T}}\cdots a_{n}^{\mathrm{T}}]^{\mathrm{T}},$ where $a_{1},\cdots
,a_{n}\ $are scalars/vectors/matrices of appropriate dimensions.

\section{\label{lyap}Preliminaries and Problem Statement}

We consider the discrete-time linear time-invariant (LTI) model of the plant:%
\begin{equation}
\left\{
\begin{array}{l}
x_{k+1}=Ax_{k}+Bd_{k}+Gw_{k} \\
y_{k}=Cx_{k}+Dd_{k}+v_{k}%
\end{array}%
\right. ,  \label{plant}
\end{equation}%
where $x_{k}\in \mathbf{R}^{n},$ $d_{k}\in \mathbf{R}^{q}$, and $y_{k}\in
\mathbf{R}^{p}$ are the state, the unknown input, and the output,
respectively; $w_{k}\in \mathbf{R}^{g}$ and $v_{k}\in \mathbf{R}^{p}$
represent zero-mean uncorrelated Gaussian process and measurement noises
with covariances $Q$ and $R$, respectively; $A,B,G,C,$ and $D$ are real and
known matrices with appropriate dimensions (When the noise shaping matrix $G$
is unknown, one needs to identify $\overline{Q}=GQG^{\mathrm{T}}$ and $R$.
The analysis of this paper can be directly extended to this case); the pair $%
(A,C)$ is assumed to be detectable; we also assume that the initial state $%
x_{0}$ is independent of the noises. Without loss of generality, we assume $%
n\geq g$ and $G\in \mathbf{R}^{n\times g}$ to be of full column rank (when
this is not the case, one can remodel the system to obtain a full rank
shaping matrix $\overline{G}$). For system (\ref{plant}), a major question
of interest is the existence condition of an observer/filter that can
estimate the state/unknown input with asymptotically stable error, using
only the output. To address this question, concepts such as strong
detectability and strong estimator have been rigorously discussed in \cite%
{Hautus1983}. As remarked in \cite{Hautus1983}, the term \textquotedblleft
strong\textquotedblright\ is to emphasize that state estimate has to be
obtained without knowing $d_{k}$.

\begin{theorem}
\label{condition}(\cite{Hautus1983}) The following statements hold true: (i)
the system (\ref{plant}) has a strong estimator if and only if it is
strongly detectable; (ii) the system (\ref{plant}) is strongly detectable if
and only if%
\begin{equation}
rank\left( \left[
\begin{array}{cc}
CB & D \\
D & 0%
\end{array}%
\right] \right) =rank(D)+rank\left( \left[
\begin{array}{c}
B \\
D%
\end{array}%
\right] \right) ,  \label{rankmatching}
\end{equation}%
and all its invariant zeros are stable, i.e.,%
\begin{equation}
rank\left( \underset{\mathcal{M}\left( z\right) }{\underbrace{\left[
\begin{array}{cc}
zI_{n}-A & -B \\
C & D%
\end{array}%
\right] }}\right) =n+rank\left( \left[
\begin{array}{c}
B \\
D%
\end{array}%
\right] \right) ,  \label{miniphase}
\end{equation}%
for all $z\in \mathbb{C}\ $and $\left\vert z\right\vert \geq 1$.
\end{theorem}

Conditions (\ref{rankmatching})-(\ref{miniphase}) are the so-called rank
matching and minimum phase requirements, respectively. Note that Theorem \ref%
{condition} holds for both the deterministic and stochastic cases (hence we
use \textquotedblleft estimator\textquotedblright\ instead of KF/full state
Luenberger observer). For system (\ref{plant}), the noise covariances $Q$
and $R$ are usually not available, and have to be identified from data.
However, all existing filtering methods in the literature adopt the
assumption of knowing $Q$ and $R$ exactly, which is not practical. This
raises the question of whether and under what conditions one can identify $Q$
and/or $R$. Of particular interest in this paper is the correlation-based
ALS method. Especially, the questions of interests are formally stated as
follows.

\begin{problem}
\label{problem1}Given system (\ref{plant}) with $A,B,G,C,$ and $D$ known, we
aim to investigate the following questions: under the assumption that system
(\ref{plant}) satisfies the strong detectability condition in Theorem \ref%
{condition}, using the ALS approach, whether and under what conditions\ one
can (a) uniquely jointly identify $Q$ and $R$; (b) uniquely identify $Q$ or $%
R$, assuming the other covariance to be known.
\end{problem}

\section{\label{ALS_framework}Identifiability Analysis of Q/R Using the ALS Framework}

This section contains our major results. In particular, we will prove that
(i) $Q$ and $R$ cannot be uniquely jointly identified; (ii) neither $Q$ nor $%
R$ is not uniquely identifiable, when the other is known.

\subsection{The filter and the choice of innovation model}

When the system (\ref{plant}) is strongly detectable, and $Q$ and $R$ are
known, one can design an unbiased and optimal filter for estimating the
state/unknown input. When $Q/R$ are not known, one can still design an
unbiased and stable (but not optimal) filter. To do so for the system (\ref%
{plant}), we assume $rank(D)=q$, and adopt the framework of \cite%
{Gillijns2007B}. Other methods in \cite{Darouach2003}, \cite{Hsieh2009}-\cite%
{Frazzoli2016}, etc., can be considered similarly and we will not elaborate
these extensions further. To be specific, the proposed filter implements the
following steps recursively after initialization:

\begin{enumerate}
\item[\textbf{1.}] Unknown input estimation:%
\begin{equation}
\widehat{d}_{k}=F(y_{k}-C\widehat{x}_{k\mid k-1});  \label{input_estimation}
\end{equation}

\item[\textbf{2.}] Measurement update:%
\begin{equation}
\widehat{x}_{k\mid k}=\widehat{x}_{k\mid k-1}+L(y_{k}-C\widehat{x}_{k\mid
k-1});  \label{Mea_update}
\end{equation}

\item[\textbf{3.}] Time update:%
\begin{equation}
\widehat{x}_{k+1\mid k}=A\widehat{x}_{k\mid k}+B\widehat{d}_{k}.
\label{prediction}
\end{equation}
\end{enumerate}

It was shown in \cite{Gillijns2007B} that steps in (\ref{input_estimation})
and (\ref{prediction}) give unbiased estimates of the unknown input and
state, respectively, if and only if the initial guess of the state is
unbiased, and $F\in \mathbf{R}^{q\times p}$ and $L\in \mathbf{R}^{n\times p}$
satisfy%
\begin{equation}
\left[ F,L\right] D=\left[ I_{q},0\right] .  \label{unbiased_condition}
\end{equation}%
The stability of the filter (\ref{input_estimation})-(\ref{prediction}) can
also be guaranteed under conditions (\ref{rankmatching})-(\ref{miniphase})
(see, e.g., \cite{Fang2012}, \cite{Kong2020Auto}). Define%
\begin{equation}
\begin{array}{l}
\widetilde{y}_{k}=y_{k}-C\widehat{x}_{k\mid k-1},\text{ }\widetilde{d}%
_{k}=d_{k}-\widehat{d}_{k}, \\
\widetilde{x}_{k\mid k}=x_{k}-\widehat{x}_{k\mid k},\text{ }\widetilde{x}%
_{k+1\mid k}=x_{k+1}-\widehat{x}_{k+1\mid k},%
\end{array}
\label{innovations}
\end{equation}%
as the innovation, the unknown input estimation error, the filtered state
error, and the state prediction error, respectively. Based on (\ref{plant})
and (\ref{input_estimation})-(\ref{prediction}), we have%
\begin{equation}
\begin{array}{l}
\widetilde{y}_{k}=Dd_{k}+C\widetilde{x}_{k\mid k-1}+v_{k}, \\
\widetilde{d}_{k}=-FC\widetilde{x}_{k\mid k-1}-Fv_{k}.%
\end{array}
\label{innovation}
\end{equation}%
The state-space model to be used for computing the autocovariance in the
next subsection is given as follows:%
\begin{equation}
\left\{
\begin{array}{l}
\widetilde{x}_{k+1\mid k}=A_{c}\widetilde{x}_{k\mid k-1}+\widetilde{G}%
\underset{\widetilde{w}_{k}}{\underbrace{\left[ w_{k},v_{k}\right] }} \\
\mathcal{Y}_{k}=L\widetilde{y}_{k}=\widetilde{L}\widetilde{x}_{k\mid
k-1}+Lv_{k}%
\end{array}%
\right. ,  \label{error_dynamics}
\end{equation}%
where $\widetilde{L}=LC$, $A_{c}=A-KC$, $\text{ }\widetilde{G}=\left[
\begin{array}{cc}
G & -K%
\end{array}%
\right]$, $\text{ }K=AL+BF$. It is worthwhile remarking that we have chosen $%
\mathcal{Y}_{k}$ as the output of the error dynamics model in (\ref%
{error_dynamics}), rather than $\widetilde{y}_{k}$, and $\widetilde{d}_{k}$
in (\ref{innovation}). This is because $\widetilde{y}_{k}$ is affected by
the unknown inputs, for which we do not assume to have any statistical
properties; $\widetilde{d}_{k}$ is the unknown input estimation error, which
can be used for analysis, but cannot be obtained from data processing (since
we do not know the true $d_{k}$). On the contrary, $\mathcal{Y}_{k}$ is a
linear transformation of the standard innovation $\widetilde{y}_{k}$, and
decoupled from $d_{k}$, thus can be obtained from data processing.

\subsection{Q and R are not uniquely jointly identifiable}

When the filter in (\ref{input_estimation})-(\ref{prediction}) is stable,
the steady-state estimation error covariance $P$ satisfies%
\begin{equation}
P=A_{c}PA_{c}^{\mathrm{T}}+GQG^{\mathrm{T}}+KRK^{\mathrm{T}}.
\label{covalya}
\end{equation}%
The autocovariance is defined as the expectation of the data with its lagged
version of itself, i.e., $\mathscr{E}(\mathcal{Y}_{k}\mathcal{Y}_{k+j}^{%
\mathrm{T}})$. Suppose that $\left\{ \mathcal{Y}_{1},\mathcal{Y}_{2},\cdots ,%
\mathcal{Y}_{N_{d}}\right\} $ is the innovations calculated from (\ref%
{error_dynamics}), where $N_{d}$ is the number of data points. Denote%
\begin{equation*}
\mathcal{Y}=\left[
\begin{array}{cccc}
\mathcal{Y}_{1} & \mathcal{Y}_{2} & \cdots & \mathcal{Y}_{N_{d}-N+1} \\
\mathcal{Y}_{2} & \mathcal{Y}_{3} & \cdots & \mathcal{Y}_{N_{d}-N+2} \\
\vdots & \vdots & \vdots & \vdots \\
\mathcal{Y}_{N} & \mathcal{Y}_{N+1} & \cdots & \mathcal{Y}_{N_{d}}%
\end{array}%
\right] \in \mathbf{R}^{Nn\times \widetilde{n}},
\end{equation*}%
where $\widetilde{n}=N_{d}-N+1$, $N$ represents the window size used in
computing the autocovariance. If the steady-state distribution of $%
\widetilde{x}_{k\mid k-1}$ is used as the initial condition, with (\ref%
{error_dynamics}), we can obtain%
\begin{equation}
\begin{array}{l}
\mathscr{E}(\mathcal{Y}_{k}\mathcal{Y}_{k}^{\mathrm{T}})=\widetilde{L}P%
\widetilde{L}^{\mathrm{T}}+LRL^{\mathrm{T}}\in \mathbf{R}^{n\times n}, \\
\mathscr{E}(\mathcal{Y}_{k+j}\mathcal{Y}_{k}^{\mathrm{T}})=\widetilde{L}%
A_{c}^{j}PL^{\mathrm{T}}-\widetilde{L}A_{c}^{j-1}KRL^{\mathrm{T}},%
\end{array}
\label{autoc}
\end{equation}%
with $j\geq 1,$and $\widetilde{L}$ being defined in (\ref{error_dynamics}).
Note that the above autocovariance expressions are independent of $k$.

Consider the first column block of the full autocovariance matrix of the
innovation process over a window of length $N$:%
\begin{equation}
\begin{array}{l}
\mathcal{R}(N)=\mathscr{E}\left( \left[
\begin{array}{c}
\mathcal{Y}_{k}\mathcal{Y}_{k}^{\mathrm{T}} \\
\vdots \\
\mathcal{Y}_{k+N-1}\mathcal{Y}_{k}^{\mathrm{T}}%
\end{array}%
\right] \right) \\
=\underset{\Theta }{\underbrace{\left[
\begin{array}{c}
\widetilde{L} \\
\Theta _{1}A_{c}%
\end{array}%
\right] }}P\widetilde{L}^{\mathrm{T}}+\underset{\Upsilon }{\underbrace{\left[
\begin{array}{c}
L \\
-\Theta _{1}K%
\end{array}%
\right] }}RL^{\mathrm{T}}\text{, }%
\end{array}
\label{expect}
\end{equation}%
where $\Theta _{1}=\left[ \widetilde{L},\widetilde{L}A_{c},\cdots ,%
\widetilde{L}A_{c}^{N-2}\right] .$ The above formula is the ideal way to
compute $\mathcal{R}(N)$. Since the process in (\ref{error_dynamics}) is
driven by Gaussian noises, it is ergodic, and a practical way of
approximating the expectation is to use the time average%
\begin{equation}
\mathcal{R}^{\ast }(N)=\frac{1}{N_{d}-N+1}\mathcal{YY}_{\func{row}}^{\mathrm{%
T}},  \label{approx}
\end{equation}%
where $\mathcal{Y}_{\func{row}}=\left[
\begin{array}{cccc}
I_{n} & 0 & \cdots & 0%
\end{array}%
\right] \mathcal{Y}$. Define $X_{s}$ as the outcome of applying the
vectorization operator to matrix $X$. Denote
\begin{equation}
b=(\mathcal{R}^{\ast }(N))_{s}.  \label{bihat}
\end{equation}%
In the following, we employ standard definition and properties of the
Kronecker product. We apply the vec operator on both sides of (\ref{expect})
and use the fact $(AXB)_{s}$ $=(B^{\mathrm{T}}\otimes A)X_{s}$ to obtain%
\begin{equation}
(\mathcal{R}(N))_{s}=(\widetilde{L}\otimes \Theta )P_{s}+(L\otimes \Upsilon
)R_{s}.  \label{bi}
\end{equation}

From (\ref{covalya}), one has%
\begin{equation}
P_{s}=(A_{c}\otimes A_{c})P_{s}+(G\otimes G)Q_{s}+(K\otimes K)R_{s}.
\label{ps}
\end{equation}%
By following similar steps in \cite{Rawlings2006}-\cite{Rawlings2009AUTO},
we then have the ALS problem formulation for identifying $Q\ $and $R$:%
\begin{equation}
\Xi ^{\ast }=\arg \min\limits_{\Xi }\left\Vert \mathcal{H}\Xi -b\right\Vert
_{\mathcal{W}}^{2},  \label{dals}
\end{equation}%
where $b$ is defined in (\ref{bihat}), $\mathcal{W}>0$ is a weighting
matrix, and%
\begin{equation}
\mathcal{H}=\left[
\begin{array}{cc}
\mathcal{H}_{1} & \mathcal{H}_{2}%
\end{array}%
\right] ,\text{ }\Xi =\left[
\begin{array}{cc}
Q_{s}^{\mathrm{T}} & R_{s}^{\mathrm{T}}%
\end{array}%
\right] ^{\mathrm{T}},  \label{RSAI}
\end{equation}%
in which,%
\begin{equation*}
\begin{array}{l}
\mathcal{H}_{1}=(\widetilde{L}\otimes \Theta )(I_{n^{2}}-A_{c}\otimes
A_{c})^{-1}(G\otimes G), \\
\mathcal{H}_{2}=(\widetilde{L}\otimes \Theta )(I_{n^{2}}-A_{c}\otimes
A_{c})^{-1}(K\otimes K)+(L\otimes \Upsilon ).%
\end{array}%
\end{equation*}%
If $\mathcal{H}$ in (\ref{RSAI}) is of full column rank, then the solution
of (\ref{dals}) exists, and is unique; moreover, desirable properties such
as unbiasedness, asymptotic convergence of the covariance estimates, as
established in \cite{Rawlings2006}-\cite{Rawlings2009AUTO}, can be obtained.
However, as it will be shown next, it is impossible for $\mathcal{H}$ in (%
\ref{RSAI}) to be of full column rank.

\begin{theorem}
\label{ALS_solution}Assume that $N\geq n+1$; the system (\ref{plant})
satisfies the strong detectability condition in Theorem \ref{condition}, and
a filter of the form (\ref{input_estimation})-(\ref{prediction}) has been
designed for it. Then the following statements hold true:

(i) $\mathcal{H}$ is of full column rank if and only if%
\begin{equation}
rank(A)=rank(C)=n\text{, }rank(L)=p;  \label{unique_condition}
\end{equation}%
(ii) it is impossible for condition (\ref{unique_condition}) to hold, i.e.,
the ALS problem (\ref{dals}) does not have a unique solution.
\end{theorem}

\begin{proof}
(i) This part follows from similar arguments with those of \cite[Chap. 3]%
{Arnold2020}, and is included here for completeness. Suppose that $\mathcal{H%
}\Xi =0,$ with $\Xi $ being defined in (\ref{RSAI}). An equivalent
representation of $\mathcal{H}\Xi =0$ is $\Theta P\widetilde{L}^{\mathrm{T}%
}+\Upsilon RL^{\mathrm{T}}=0,$ which can be further expressed as%
\begin{equation}
\widetilde{L}P\widetilde{L}^{\mathrm{T}}+LRL^{\mathrm{T}}=0,\text{ }\Theta
_{1}A_{c}PC^{\mathrm{T}}L^{\mathrm{T}}=\Theta _{1}KRL^{\mathrm{T}},
\label{two_equa}
\end{equation}%
where $\Theta _{1}$ is defined in (\ref{expect}). Given that the system (\ref%
{plant}) satisfies the strong detectability condition in Theorem \ref%
{condition}, if one designs a filter of the form (\ref{input_estimation})-(%
\ref{prediction}) for it, one has that $A_{c}=A-KC$ is Schur stable (see,
e.g., \cite{Kong2020Auto}). Hence, $(A_{c},\widetilde{L})$ is detectable,
because $A_{c}=A_{c}-0\cdot \widetilde{L}$ is stable. It follows that $%
\Theta _{1}$ is of full column rank, when $N\geq n+1$. We first prove
sufficiency. Based on (\ref{two_equa}), we have%
\begin{equation*}
R=-CPC^{\mathrm{T}},\text{ }A_{c}PC^{\mathrm{T}}-KR=0,
\end{equation*}%
when $L$ is of full rank. The above two equalities lead to%
\begin{equation}
APC^{\mathrm{T}}=0\text{.}  \label{condition_zero}
\end{equation}%
If $A$ is nonsingular and $C$ is of full column rank, then we have $P=0$, $%
R=0$, which further implies that $Q=0$. As a result, the null space of $%
\mathcal{H}$ only has one element, i.e., the zero vector. In other words, $%
\mathcal{H}$ \ is of full column rank. We next prove necessity. Firstly,
assume that $A$ is singular, i.e., there exists a nonzero vector $z$ such
that $Az=0$. Set $P=zz^{\mathrm{T}}$ so that (\ref{condition_zero}) holds.
If we select $R=-CPC^{\mathrm{T}}$, then one has that the two equalities in (%
\ref{two_equa}) hold. Moreover, from (\ref{ps}), one has that%
\begin{equation*}
\begin{array}{l}
\underset{\mathcal{M}}{\underbrace{(G\otimes G)}}Q_{s}=\underset{\mu }{%
\underbrace{(I_{n^{2}}-A_{c}\otimes A_{c})P_{s}-(K\otimes K)R_{s}}}, \\
\Rightarrow Q_{s}=\mathcal{M}^{\mathrm{T}}(\mathcal{MM}^{\mathrm{T}}\mathcal{%
)}^{-1}\mu ,%
\end{array}%
\end{equation*}%
where we have used the assumption that $G$ (and hence $\mathcal{M)}$ is of
full column rank. Note that, $\mu $ might or might not be zero. However, for
either case, we have $R_{s}\neq 0$. Hence, there exists a nonzero element $%
\Xi =\left[
\begin{array}{cc}
Q_{s}^{\mathrm{T}} & R_{s}^{\mathrm{T}}%
\end{array}%
\right] ^{\mathrm{T}}$ in the null space of $\mathcal{H}$. Now assume that $%
rank(L)\neq p$ and there exists a nonzero vector $z_{1}$ such that $Lz_{1}=0$%
. If we set $R=z_{1}z_{1}^{\mathrm{T}}$, $P=0$, then the two equalities in (%
\ref{two_equa}) hold. Then by following similar arguments with the above,
there exists a nonzero element $\Xi =\left[
\begin{array}{cc}
Q_{s}^{\mathrm{T}} & R_{s}^{\mathrm{T}}%
\end{array}%
\right] ^{\mathrm{T}}$ in the null space of $\mathcal{H}$. The necessity of
full column rankness of $C$ can also be proved similarly.

(ii) Assume that condition (\ref{unique_condition}) holds.
For the unbiased filter design, one must have $LD=0\Rightarrow D=0,$ which
contradicts the assumption that $D$ is of full column rank. As a result, it
is impossible for condition (\ref{input_estimation}) to hold, and the ALS
problem (\ref{dals}) does not have a unique solution. This completes the
proof.
\end{proof}

\begin{remark}
\textrm{\textrm{\label{rm1}The results in Theorem \ref{ALS_solution} can be
considered as generalizations of those in \cite{Rawlings2006}-\cite%
{Rawlings2015} to the case with unknown inputs. Although Theorem \ref%
{ALS_solution} is established for the case with direct feedthrough, for the
case without feedthrough, i.e., $D=0$ in (\ref{plant}), the same statements
regarding the non-identifiability of $Q$ and $R$ can also be obtained (this
can be shown by taking the filtering framework of \cite{Gillijns2007A} and
similar steps above). For the situation without feedthrough, the innovation
model to be used for calculating the autocovariance is different from (\ref%
{error_dynamics}), since in this case only the filtered state error is
unbiased (i.e., the predicted state error is biased). Detailed proofs are
omitted here due to limited space.} }
\end{remark}

\subsection{Neither Q nor R is uniquely identifiable when the other is known}

We next consider part (b) of Problem \ref{problem1}, i.e., identifiability
of $Q$ or $R$ when the other is known. We firstly consider the case of
estimating $Q$ when $R$ is known. Denote%
\begin{equation}
b_{Q}=b-\mathcal{H}_{2}R_{s},  \label{b_Q_ALS}
\end{equation}%
where $b$ and $\mathcal{H}_{2}$ are defined in (\ref{bihat}) and (\ref{RSAI}%
), respectively. By following similar steps to those of the previous
subsection, we have the following ALS problem formulation for identifying $Q$%
:%
\begin{equation}
\Xi _{Q}^{\ast }=\arg \min\limits_{\Xi _{Q}}\left\Vert \mathcal{H}_{1}\Xi
_{Q}-b_{Q}\right\Vert _{\mathcal{W}_{Q}}^{2}  \label{Q_unknown_ALS}
\end{equation}%
where $\Xi _{Q}=Q_{s},$ $\mathcal{H}_{1}$ is defined in (\ref{dals}), $b_{Q}$
is defined in (\ref{b_Q_ALS}), and $\mathcal{W}_{Q}>0$. For the case of
estimating $R$ when $Q$ is known, we denote%
\begin{equation}
b_{R}=b-\mathcal{H}_{1}Q_{s},  \label{b_R_ALS}
\end{equation}%
where $b$ and $\mathcal{H}_{1}$ are defined in (\ref{bihat}) and (\ref{RSAI}%
), respectively. By following similar steps to those of the previous
subsection, we then have the following ALS problem formulation for
identifying $R$:%
\begin{equation}
\Xi _{R}^{\ast }=\arg \min\limits_{\Xi _{R}}\left\Vert \mathcal{H}_{2}\Xi
_{R}-b_{R}\right\Vert _{\mathcal{W}_{R}}^{2}  \label{R_unknown_ALS}
\end{equation}%
where $\Xi _{R}=R_{s},$ $\mathcal{H}_{2}$ is defined in (\ref{dals}), $b_{R}$
is defined in (\ref{b_Q_ALS}), and $\mathcal{W}_{R}>0$. We then have the
following result.

\begin{corollary}
\label{Q_unknown_ALS_extension}Assume that $N\geq n+1$; the system (\ref%
{plant}) satisfies the strong detectability condition in Theorem \ref%
{condition}, and a filter of the form (\ref{input_estimation})-(\ref%
{prediction}) has been designed for it. Then neither of the ALS problems (%
\ref{Q_unknown_ALS}) and (\ref{R_unknown_ALS}) has a unique solution; in
other words, neither $Q$ nor $R$ is uniquely identifiable when the other is
known.
\end{corollary}

\begin{proof}
For the ALS problems (\ref{Q_unknown_ALS}) and (\ref{R_unknown_ALS}) to have
a unique solution, it is necessary that $\mathcal{H}_{1}$ and $\mathcal{H}%
_{2}$ are of full column rank, respectively. From Theorem \ref{ALS_solution}%
, we know that $rank(L)=p$ is a necessary condition for the full column
rankness of both $\mathcal{H}_{1}$ and $\mathcal{H}_{2}$. However, as
discussed in the proof of Theorem \ref{ALS_solution}, $L$ cannot have full
column rank. In other words, $\mathcal{H}_{1}$ and $\mathcal{H}_{2}$ cannot
be of full column rank. This completes the proof.
\end{proof}

Corollary \ref{Q_unknown_ALS_extension}, together with Theorem \ref%
{ALS_solution}, provide a complete and negative answer to Problem \ref%
{problem1}. For the above-mentioned scenarios where the ALS problems are
ill-posed and the non-identifiability of noise covariances is obtained, a
natural idea is to use regularization to introduce further constraints to
uniquely determine the solution \cite{Chen2013}. However, a key question to
be answered is whether some desirable properties can be guaranteed for the
covariance estimates. A full investigation of the above questions is a
subject of our current and future work. The results in this paper can be
readily extended to the case with correlated noises. They can also be
generalized to other more complex scenarios, e.g., linear time varying
systems as in \cite{Ge2017}, although the solution uniqueness conditions of
the corresponding least-squares will be harder to analyze.

\section{\label{exten_delay}Examples}

In this section, we use some numerical examples to verify the theoretical
results. For simplicity, we take $G=I_{n}$ in the plant model (\ref{plant}).
Assume that the plant model (\ref{plant}) is has the following system
matrices%
\begin{equation*}
A=\left[
\begin{array}{ll}
1 & 1 \\
0 & 1%
\end{array}%
\right] ,\text{ }C=\left[
\begin{array}{ll}
1 & 2 \\
1 & 1%
\end{array}%
\right] ,\text{ }B=\left[
\begin{array}{l}
0 \\
1%
\end{array}%
\right] ,\text{ }D=\left[
\begin{array}{l}
1 \\
0%
\end{array}%
\right] .
\end{equation*}%
It can be easily verified that conditions (\ref{rankmatching})-(\ref%
{miniphase}) are satisfied. To design the filter in (\ref{input_estimation}%
)-(\ref{prediction}), we select%
\begin{equation*}
F=\left[
\begin{array}{ll}
1 & 0.5%
\end{array}%
\right] ,\text{ }L=\left[
\begin{array}{ll}
0 & -1 \\
0 & 0%
\end{array}%
\right] ,
\end{equation*}%
such that condition (\ref{unbiased_condition}) is met and $A_{c}$ in (\ref%
{error_dynamics}) is stable. Note that the requirements in (\ref%
{unique_condition}) are all satisfied except that $L$ is not of full column
rank (because it cannot be, as proved in Theorem \ref{ALS_solution}). Hence,
$\mathcal{H}$ in (\ref{dals}) cannot be of full column rank. To double
confirm, select $N=10$, it can be checked that $rank(\mathcal{H})=2,$ $rank(%
\mathcal{H}_{1})=1,$ $rank(\mathcal{H}_{2})=2,$ where, $\mathcal{H}\in
\mathbf{R}^{40\times 8},$ $\mathcal{H}_{1}\in \mathbf{R}^{40\times 4}$ and $%
\mathcal{H}_{2}\in \mathbf{R}^{40\times 4}$. This validates Theorem \ref%
{ALS_solution} and Corollary \ref{Q_unknown_ALS_extension}.

\section{\label{conclusion}Conclusions}

The past few decades have witnessed much progress in optimal filtering for
stochastic systems with arbitrary unknown inputs and Gaussian noises.
However, the existing works assume perfect knowledge of the noise
covariances in the filter design, which is impractical. In this paper, for
linear Gaussian systems under unknown inputs, we have investigated the
identifiability question of the process and measurement noises covariances
(i.e., $Q$ and $R$) using the correlation-based ALS method. In particular,
we have shown that the ALS problem for estimating $Q$ and/or $R$ does not
have a unique solution. The above findings reveal that the noise covariances
are in general not uniquely identifiable using the ALS approach. This not
only helps us to have a better understanding of the applicability of
existing filtering frameworks under unknown inputs (since almost all of them
require perfect knowledge of the noise covariances) but also calls for
further investigation of alternative and more viable noise covariance
methods under unknown inputs. Especially, it remains to be explored whether
the noise covariances are uniquely identifiable using other
correlation-based methods. We are also interested to use regularization for
noise covariance estimation under unknown inputs, and investigate the relevant property guarantees for the covariance estimates. The above topics are the main subject of our
current and future work.

\section{Acknowledgment}

The authors would like to thank the reviewers and Editors for their
constructive suggestions which have helped to improve the quality and
presentation of this paper significantly.

\end{document}